\documentclass[10pt,reqno]{amsart}
\usepackage{bbm}
\usepackage{mathrsfs}
\usepackage{cite}
\usepackage{amsfonts} %%% i.e. use 12pt type
\usepackage[dvipsnames,usenames]{color}
\usepackage{graphicx,latexsym,bm,amsmath,amssymb,verbatim,multicol,lscape}
\newtheorem{thm}{Theorem} [section]
\newtheorem{cor}[thm]{Corollary}

\newtheorem{lem}[thm]{Lemma}

\theoremstyle{definition}

\theoremstyle{remark}

\numberwithin{equation}{section}
\begin{document}
\title{The 2-adic valuations of Stirling numbers of the first kind }
\begin{abstract}
Let $n$ and $k$ be positive integers. We denote by $v_2(n)$ the
2-adic valuation of $n$. The Stirling numbers of the first kind,
denoted by $s(n,k)$, counts the number of permutations of $n$ elements
with $k$ disjoint cycles. In recent years, Lengyel, Komatsu and Young,
Leonetti and Sanna, and Adelberg made some progress on the
$p$-adic valuations of $s(n,k)$. In this paper, by introducing the
concept of $m$-th Stirling numbers of the first kind and providing
a detailed 2-adic analysis, we show an explicit formula on the
2-adic valuation of $s(2^n, k)$. We also prove that
$v_2(s(2^n+1,k+1))=v_2(s(2^n,k))$ holds for all
integers $k$ between 1 and $2^n$. As a corollary, we show
that $v_2(s(2^n,2^n-k))=2n-2-v_2(k-1)$ if $k$ is odd and
$2\le k\le 2^{n-1}+1$. This confirms partially a conjecture
of Lengyel raised in 2015. Furthermore, we show that if
$k\le 2^n$, then $v_2(s(2^n,k)) \le v_2(s(2^n,1))$ and
$v_2(H(2^n,k))\leq -n$, where $H(n,k)$ stands for the
$k$-th elementary symmetric functions of $1,1/2,...,1/n$.
The latter one supports the conjecture of
Leonetti and Sanna suggested in 2017.
\end{abstract}

\author[M. Qiu]{Min Qiu}
\address{Mathematical College, Sichuan University, Chengdu 610064, P.R. China}
\email{qiumin126@126.com}
\author[S.F. Hong]{Shaofang Hong$^*$}
\address{Mathematical College, Sichuan University, Chengdu 610064, P.R. China}
%\curraddr{}
\email{sfhong@scu.edu.cn; s-f.hong@tom.com; hongsf02@yahoo.com}
\thanks{$^*$S.F. Hong is the corresponding author and was supported
partially by National Science Foundation of China Grant \#11771304
and by the Fundamental Research Funds for the Central Universities.}
\keywords{2-adic valuation, 2-adic analysis, Stirling number of the
first kind, the $m$-th Stirling number of the first kind,
elementary symmetric function.}
\subjclass[2000]{Primary 11B73, 11A07}
\maketitle

\section{Introduction}
Let $n$ and $k$ be nonnegative integers. Define the {\it Pochhammer
symbol} $(x)_n$ by $(x)_n:=x(x+1)(x+2)\cdots(x+n-1)$ if $n\ge 1$ and
$(x)_n:=1$ if $n=0$. The \emph{Stirling numbers of the first kind},
denoted by $s(n,k)$, counts the number of permutations of $n$
elements with $k$ disjoint cycles. One can characterize $s(n,k)$
by
$$(x)_n=\sum_{k=0}^{n}s(n,k)x^k. $$
The \emph{Stirling numbers of the second kind}
$S(n,k)$ is defined as the number of ways to partition a set of
$n$ elements into exactly $k$ nonempty subsets, and we have
$$ S(n,k)=\frac{1}{k!}\sum_{i=0}^{k}(-1)^i\binom{k}{i}(k-i)^n.$$

Divisibility properties of integer sequences have long been objects of
interest in number theory. Many authors studied the divisibility
properties of Stirling numbers of the second kind, see, for example,
\cite{[Ad],[Am],[CF],[Da],[HZ1],[TL],[TL2],[TL3],[Lu],[Wan],[HZ2],[HZ3]}.
But there are few of results on the divisibility of the 
Stirling numbers of the first kind in the literature. Actually, 
unlike $S(n,k)$, there is no easy way to use an explicit formula
for the Stirling numbers of the first kind. This makes it more
difficult to investigate the divisibility properties of $s(n,k)$.

On the other hand, for any positive integer $n$ and $k$ with 
$n \geq k$, the Stirling numbers of the first kind is closely 
related to $H(n,k)$ by the following identity
\begin{align}\label{1.1}
s(n+1,k+1)=n!H(n,k)
\end{align}
(see Lemma 1.1 in \cite{[LS]}),
where $H(n,k)$ represents the {\it $k$-th elementary symmetric functions}
of $1,1/2,...,1/n$, that is,
$$ H(n,k) := \sum_{1 \leq i_1 < \cdots < i_k \leq n} \frac{ 1 } { i_1 \cdots i_k }. $$
As usual, for any prime $p$ and for any integer $n$,
we let $v_p(n)$ stands for the {\it $p$-adic valuation}
of $n$, i.e., $v_p(n)$ is the biggest nonnegative integer
$r$ with $p^r$ dividing $n$. If $x=\frac{n_1}{n_2}$,
where $n_1$ and $n_2$ are integers and $n_2\ne 0$,
then we define $v_p(x):=v_p(n_1)-v_p(n_2)$.
The integrality problem on $H(n,k)$, studied by Theisinger \cite{[T]}, Nagell \cite{[N]}
and Erd\H{o}s and Niven \cite{[EN]}, and recently by Chen and Tang \cite{[CT]},
Hong and Wang \cite{[HW]} \cite{[WH]} as well as Luo, Hong, Qian and Wang \cite{[LHQW]},
is close to the $p$-adic valuation of $H(n,k)$. However, the well-known Legendre
formula about the $p$-adic valuation of the factorial tells us that
$$v_p(n!)=\frac{n-d_p(n)}{p-1},$$
where $d_p(n)$ represents the base $p$ digital sum of $n$.
Hence the investigation of $v_p(H(n,k))$  is equivalent to
the investigation of $v_p(s(n+1,k+1))$.

Let $p$ be a prime and $n$ be a positive integer. In recent years, 
some progress on $p$-adic valuations of $s(n,k)$ were made by several
authors. In \cite{[L]}, Lengyel showed
that, for any positive integer $k$, the $p$-adic valuation of $s(n,k)$ goes
to infinity when $n$ approaches infinity. Moreover, Lengyel \cite{[L]} proved
that there exists a constant $c'=c'(k,p)> 0$ so that for any $n\ge n_0(k,p)$,
one has $v_p(s(n,k))\ge c'n$. This implies that the $p$-adic valuation of
$H(n,k)$ has a lower bound. Consequently, Leonetti and Sanna \cite{[LS]}
conjectured that there exists a positive constant $c=c(k,p)$ such that
\begin{align}
v_p(H(n,k))<-c\log n\label{1.2}
\end{align}
for all large $n$ and confirmed this conjecture for some special cases.
Furthermore, if $k$ is a nonnegative integer and $n$ is of the form $n=kp^r+m$
with $0\le m<p^r$, then Komatsu and Young proved in \cite{[KY]} that
$v_p(s(n+1,k+1))=v_p(n!)-v_p(k!)-kr$.
Recently, using the study of the higher order Bernoulli numbers $B_n^{(l)}$,
Adelberg \cite{[Ad]} investigated some $p$-adic properties of Stirling
numbers of both kinds. These motivate us to further study the $p$-adic
valuation of $s(n,k)$. In this paper, we are mainly concerned with
the $2$-adic valuation of the Stirling numbers of the first kind.

We refer the readers to \cite{[HZ1],[TL],[TL2],[TL3],[Wan],[HZ2],[HZ3]}
on some results on the 2-adic valuation of $S(n,k)$. In 1994,
Lengyel \cite{[TL]} conjectured, proved by Wannemacker \cite{[Wan]}
in 2005, the 2-adic valuation of $S(2^n,k)$ equals $d_2(k)-1$.
However, one finds that the 2-adic valuation of $s(2^n,k)$ 
is much more complicated than the second kind case. 
Lengyel \cite{[L]} proved that $v_2(2^n, 2))=2^n-2n$ and  
$v_2(2^n, 3))=2^n-3n+3$. Komatsu and Young \cite{[KY]} 
used the theory of Newton polygon to show that
$v_2(s(2^n,2^m))=2^n-2^m(n-m+1)$ with $n$ and $m$ being 
positive integers and $n\geq m$. In this paper, by
introducing the concept of $m$-th Stirling numbers of the first kind
and supplying a detailed 2-adic analysis, we are able to evaluate
$v_2(s(2^n, k))$. Evidently, $v_2(s(2,1))=v_2(s(2,2))=v_2(s(2^n,2^n))=v_2(1)=0$
and $v_2(s(2^n,2^n-1))=v_2(\binom{2^n}{2})=n-1$.
Now let $n$ and $i$ be integers with $n\ge 2$ and $1\le i\le 2^n-2$.
Then one can write $i=2^m-k$ for $(m,k)\in T_n$, where
$$
T_n:=\{(m,k) : 2\le m\le n, 2\le k\le 2^{m-1}+1, m, k\in \mathbb{Z}\}.
$$
For any given real number $y$, by $\lfloor y\rfloor$ and $\lceil y\rceil$
we denote the largest integer no more than $y$ and the smallest integer
no less than $y$, respectively. We can now state the first
main result of this paper as follows.
\begin{thm}\label{thm1}
For any integer $n,m$ and $k$ such that $2\le m\le n$ and
$2\le k\le 2^{m-1}+1$, we have
\begin{align}\label{1.3}
v_2(s(2^n,2^m-k)) = 2^n-2^m-(n-m)\Big(2^m-2\Big\lfloor\frac{k}{2}\Big\rfloor\Big)
+ m-2-v_2\Big(\Big\lfloor\frac{k}{2}\Big\rfloor\Big)+(n-1)\epsilon_k,
\end{align}
where $\epsilon_k=0$ if $k$ is even, and $\epsilon_k=1$ if $k$ is odd.
\end{thm}

For the Stirling numbers of the second kind, Hong, Zhao and Zhao \cite{[HZ1]}
confirmed a conjecture of Amdeberhan, Manna and Moll \cite{[Am]} raised in 2008
by showing that
$$v_2(S(2^n+1,k+1))=v_2(S(2^n,k)).$$
By using Theorem 1.1, we establish the following similar result
for the Stirling numbers of the first kind which is the second 
main result of this paper.

\begin{thm}\label{thm2}
For any positive integer $n$ and $k$ such that $k\le 2^n$, we have
\begin{align}\label{1.4}
v_2(s(2^n+1,k+1))=v_2(s(2^n,k)).
\end{align}
\end{thm}

Lengyel proved in \cite{[L]} that, for any prime $p$, any integer 
$a\ge 1$ with $(a,p)=1$, and any even $k\ge 2$ with the condition
\begin{align}\label{1.5}
\exists n_1: n_1> 3\log_p{k}+\log_p{a}\ {\rm such\ that} 
\  v_p(s(ap^{n_1},ap^{n_1}-k))< n_1
\end{align}
or $k=1$ with $n_1=1$, then for $n\ge n_1$ one has
$$v_p(s(ap^{n+1},ap^{n+1}-k))=v_p(s(ap^n,ap^n-k))+1.$$
Meanwhile, Lengyel believed that (\ref{1.5}) holds for all even $k\ge 2$. 
Furthermore, for any odd $k\ge 3$, Lengyel conjectured 
in \cite{[L]} that, for any integer $n\ge n_1(p,k)$
with some sufficiently large $n_1(p,k)$, one has
\begin{align}
v_p(s(ap^{n+1},ap^{n+1}-k))=v_p(s(ap^n,ap^n-k))+2.\label{1.6}
\end{align}
Now letting $m=n$, Theorem 1.1 gives us the following result.
\begin{cor}\label{cor1}
For any integer $n$ and $k$ such that $n\ge 2$ and 
$2\le k\le 2^{n-1}+1$, we have
\begin{align*}
v_2(s(2^n,2^n-k)) =\Big\{\begin{array}{ll}n-1-v_2(k) 
& \ {\it if}\ 2\mid k, \\
2n-2-v_2(k-1) & \ {\it if}\ 2\nmid k. \end{array}
\end{align*}
\end{cor}
\noindent Evidently, Corollary \ref{cor1} infers that when $p=2$ and $a=1$
with $n_1\ge \lceil \log_2{k} \rceil+1$, Condition (\ref{1.5}) always holds 
for all even $k\ge 2$ and Conjecture (\ref{1.6}) is true. We point out
that there are two typos in equality (3.7) and the one below (3.8)
in \cite{[L]}, where $n+1$ should read as $n$.

Another consequence of Theorem \ref{thm1} is the following interesting result.
\begin{cor}\label{cor2}
For any positive integer $n$ and $k$ such that $k\le 2^n$, we have
\begin{align}\label{1.7}
v_2(s(2^n,k)) \le v_2(s(2^n,1)).
\end{align}
\end{cor}

From Theorem \ref{thm2} and Corollary \ref{cor2}, we can
derive an upper bound for $v_2(H(2^n,k))$ as follows.

\begin{cor}\label{cor3}
For any positive integer $n$ and $k$ such that $k\le 2^n$, we have
\begin{align}
v_2(H(2^n,k)) \leq -n.\label{1.8}
\end{align}
\end{cor}
\noindent Clearly, Corollary \ref{cor3} confirms partially Conjecture
(\ref{1.2}) that is due to Leonetti and Sanna and raised in \cite{[LS]}.

This paper is organized as follows. First of all, in the next section,
we introduce the concept of the $m$-th Stirling numbers of the first kind
and reveal some useful properties of Stirling numbers of the first kind.
Subsequently, we prove Theorem \ref{thm1} in Section 3, and show
Theorem \ref{thm2} in Section 4. Finally, the proofs of Corollaries
\ref{cor2} and \ref{cor3} are presented in the last section.

\section{Auxiliary results on Stirling numbers of the first kind}

Let $n$ and $k$ be positive integers. Some basic identities involving
Stirling numbers of the first kind can be listed as follows (see \cite{[LC]}):
$$s(n+1,k)=ns(n,k)+s(n,k-1),\ s(n,0)=0,\ s(n,k)=0 \ {\rm if}\ k\ge n+1,$$
$$\ s(n,1)=(n-1)!,\ s(n,n)=1,\ s(n,n-1)=\binom{n}{2},$$
$$s(n,2)=(n-1)!\sum_{k=1}^{n-1}\frac{1}{k}\ {\rm and}\
s(n,n-2)=\frac{1}{4}(3n-1)\binom{n}{3}\ {\rm if}\ n\ge 2.$$

The following lemma can be derived from the general formula for Stirling
numbers of the first kind in terms of harmonic numbers and the Pochhammer symbol.
\begin{lem}\label{lem1} \cite{[VA]}
Let $n$ and $k$ be positive integers. If $n+k$ is odd, then
$$ s(n,k)=\frac{1}{2}\sum_{i=k+1}^{n}s(n,i)\binom{i-1}{i-k}n^{i-k}(-1)^{n-i}.$$
\end{lem}

Now for any nonnegative integer $m$, we introduce the concept of the
{\it $m$-th Stirling numbers of the first kind}, denoted by $s_m(n, k)$
which are defined by the following identity:
\begin{align*}
(x+m)_n=(x+m)(x+m+1)\cdots(x+m+n-1) = \sum_{k=0}^{n} s_m(n,k)x^k,
\end{align*}
where $n$ and $k$ are nonnegative integers such that $n\ge k$.
Clearly, one has $s_0(n,k)=s(n,k)$. Furthermore, we have the
following convolution results on $s_m(n,k)$ that
play a crucial role in the proof of Theorem \ref{thm1}.

\begin{lem}\label{lem2}
Let $m, k$ and $n$ be nonnegative integers. Then
$$  s(m+n,k)=\sum_{i=0}^{k}s(m,i) s_m(n,k-i). $$
\end{lem}

\begin{proof}
First of all, by the definitions of Stirling numbers of the first kind
and of the $m$-th Stirling numbers of the first kind, we deduce that

\begin{align*}
\sum_{k=0}^{m+n}s(m+n,k)x^k &= x(x+1)\cdots(x+m-1)(x+m)\cdots(x+m+n-1)\\
&= \Big(\sum_{i=0}^{m}s(m,i)x^i\Big)\Big(\sum_{j=0}^{n}s_m(n,j)x^j\Big)\\
&=\sum_{i=0}^{m}\sum_{j=0}^{n}s(m,i)s_m(n,j)x^{i+j}\\
&=\sum_{k=0}^{m+n}\Big(\sum_{i+j=k}s(m,i)s_m(n,j)\Big)x^{k}\\
&=\sum_{k=0}^{m+n}\Big(\sum_{i=0}^{k}s(m,i)s_m(n,k-i)\Big)x^{k}.
\end{align*}
Then comparing the coefficients of $x^k$ on both sides gives us the
desired result.
\end{proof}

\begin{lem}\label{lem3}
Let $m, k$ and $n$ be nonnegative integers. Then
$$  s_m(n,k)=\sum_{i=k}^{n}s(n,i)\binom{i}{i-k}m^{i-k}.$$
\end{lem}

\begin{proof}
Using the definitions of $s_m(n,k)$ and $s(n,k)$ and noticing that $s(n, 0)=0$,
we obtain that
\begin{align*}
\sum_{k=0}^{n}s_m(n,k)x^k  &= (x+m)(x+m+1)\cdots(x+m+n-1)\\
&=\sum_{i=0}^{n}s(n,i)(x+m)^i\\
&=\sum_{i=0}^{n}s(n,i)\Big(\sum_{k=0}^{i}\binom{i}{k}x^km^{i-k}\Big)\\
&=\sum_{k=0}^{n}\Big(\sum_{i=k}^{n}s(n,i)\binom{i}{k}m^{i-k}\Big)x^k.
\end{align*}
One then compare the corresponding coefficients on both sides to get
the required result. So Lemma 2.3 is proved.
\end{proof}

Obviously, $s_m(n,k)\equiv s(n,k) \pmod m$ if $m$ is a positive
integer. In particular, we can deduce the following result.

\begin{lem}\label{lem4}
Let $n$ be an integer with $n\ge 2$. If (\ref{1.3}) holds
for all positive integers $m$ and $k$ with $(m,k)\in T_n$,
then for any integer $t$ with $1\le t\le 2^n$, we have
\begin{align}\label{2.1}
v_2(s_{2^n}(2^n,t)) = v_2(s(2^n,t))
\end{align}
and
\begin{align}\label{2.2}
v_2(s_{2^n}(2^n,t)-s(2^n,t))\ge v_2(s(2^n,t))+2.
\end{align}
\end{lem}

\begin{proof}
The case with $t=2^n$ is clearly true, since $s_{2^n}(2^n,2^n) = s(2^n,2^n)=1$.

If $t=2^n-1$, then
\begin{align}\label{2.3}
s(2^n,t)=s(2^n,2^n-1)=\binom{2^n}{2}=2^{n-1}(2^n-1).
\end{align}
By Lemma \ref{lem3}, one derives that
\begin{align}
s_{2^n}(2^n,t)&=s_{2^n}(2^n,2^n-1)=s(2^n,2^n-1)+s(2^n,2^n)2^n\binom{2^n}{1}\notag\\
&=2^{n-1}(2^n-1)+2^{2n}=2^{n-1}(3\times2^n-1).\label{2.4}
\end{align}
Thus by (\ref{2.3}) and (\ref{2.4}) and noticing that $n\ge 2$, we arrive at
$$
v_2(s_{2^n}(2^n,t)) = v_2(s(2^n,t))=v_2(s(2^n,2^n-1))=n-1
$$
and
$$
v_2(s_{2^n}(2^n,t)-s(2^n,t))=v_2(2^{2n})=2n\ge n-1+3=v_2(s(2^n,t))+3
$$
as expected. Hence (\ref{2.1}) and (\ref{2.2}) are true when $t=2^n$ and $t=2^n-1$.

In what follows, we let $1\le t\le 2^n-2$. Then one may write $t:=2^m-k$
for some $(m,k)\in T_n$. Since (\ref{1.3}) holds for $m$ and $k$, we have
\begin{align}\label{2.5}
v_2(s(2^n,t))&=v_2(s(2^n,2^m-k)) \notag \\
&= 2^n-2^m-(n-m)\Big(2^m-2\Big\lfloor\frac{k}{2}\Big\rfloor\Big)
+ m-2-v_2\Big(\Big\lfloor\frac{k}{2}\Big\rfloor\Big)+(n-1)\epsilon_k,
\end{align}
with $\epsilon_k=0$ when $k$ is even, and $\epsilon_k=1$ when $k$ is odd.

Again, by Lemma \ref{lem3}, one can write
\begin{align}
s_{2^n}(2^n,t)=s(2^n,t) + L,\label{2.6}
\end{align}
where
\begin{align} \label{2.7}
L=\sum_{i=t+1}^{2^n}L_i\ \  {\rm and}\ \ L_i=s(2^n,i)2^{n(i-t)}\binom{i}{i-t}.
\end{align}
For any integer $i$ with $t+1\le i\le 2^n$, let $\Delta_i:=v_2(L_i)-v_2(s(2^n,t)).$
We claim that
\begin{align}\label{2.8}
\Delta_i\ge 2.
\end{align}
Then from this claim, one can deduce that $v_2(L_i)\ge v_2(s(2^n,t))+2$. Hence
\begin{align}\label{2.9}
v_2(s_{2^n}(2^n,t)-s(2^n,t))=v_2(L)=v_2\Big(\sum_{i=t+1}^{2^n} L_i\Big)
\ge\min_{t+1\le i\le 2^n}\{v_2(L_i)\}\ge v_2(s(2^n,t))+2
\end{align}
as (\ref{2.2}) asserted. Furthermore, by (\ref{2.6}), (\ref{2.9})
and using the isosceles triangle principle (see, for example, \cite{[K]}), 
we derive that
$$v_2(s_{2^n}(2^n,t))=v_2(s(2^n,t)+L)=v_2(s(2^n,t))$$
as (\ref{2.1}) required. So to finish the proof of Lemma \ref{lem4}, it remains
to show claim (\ref{2.8}), which will be done in what follows.

First of all, if $i=2^n$, then $s(2^n,2^n)=1$.
So it follows from (\ref{2.7}) and (\ref{2.5}) that
\begin{align}
&\Delta_{i}\notag\\
=&v_2(L_{2^n})-v_2(s(2^n,t))\notag\\
\ge & v_2\Big(s(2^n,2^n)2^{n(2^n-t)}\Big)-v_2(s(2^n,t))=n(2^n-t)-v_2(s(2^n,t)) \label{2.10}\\
=&(n-1)2^n-(m-1)2^m+n\Big(k-2\Big\lfloor\frac{k}{2}\Big\rfloor\Big)
+m\Big(2\Big\lfloor\frac{k}{2}\Big\rfloor-1\Big)+
v_2\Big(\Big\lfloor\frac{k}{2}\Big\rfloor\Big)-(n-1)\epsilon_k+2\notag\\
\ge & n\Big(k-2\Big\lfloor\frac{k}{2}\Big\rfloor\Big)
+m\Big(2\Big\lfloor\frac{k}{2}\Big\rfloor-1\Big)+
v_2\Big(\Big\lfloor\frac{k}{2}\Big\rfloor\Big)-(n-1)\epsilon_k+2\notag\\
\ge & 4 \label{2.11}
\end{align}
since $t=2^m-k$, $n\ge m\ge 2$, $k\ge 2$, $\epsilon_k=0$ if $k$ is even and
$\epsilon_k=1$ if $k$ is odd, and also notice that if $k$ is odd,
then $\lfloor\frac{k}{2}\rfloor=\frac{k-1}{2}$ and $k\ge 3$.
Hence (\ref{2.8}) is proved when $i=2^n$.

Consequently, if $i=2^n-1$, then by (\ref{2.7}), (\ref{2.5}), (\ref{2.3}), (\ref{2.10})
and (\ref{2.11}) one obtains that
\begin{align}
\Delta_{i}&=v_2(L_{2^n-1})-v_2(s(2^n,t))\ge v_2\big(s(2^n,2^n-1)2^{n(2^n-t-1)}\big)-v_2(s(2^n,t))\notag\\
&=v_2(s(2^n,2^n-1))+n(2^n-t-1)-v_2(s(2^n,t))\notag\\
&=n(2^n-t)-v_2(s(2^n,t))-1\notag\\
&\ge 3.\notag
\end{align}
Thus claim (\ref{2.8}) is true when $i=2^n-1$.

Finally, we let $t+1\le i\le 2^n-2$ with $t=2^m-k$. Then we have the following partition:
$$
[2^m-k+1, 2^n-2]\cap\mathbb{Z}=
V\cup\Big(\bigcup_{l=m+1}^n[2^{l-1}-1, 2^l-1)\cap\mathbb{Z}\Big),
$$
where $V=[2^m-k+1, 2^m-1)\cap\mathbb{Z}$ if $k>2$, and $V$ is empty if $k=2$.
Then one may let $i:=2^l-j$ for some $(l,j)\in T_n$ with $l\ge m$.
Notice that $2\le j\le k-1$ if $l=m$.
Then by the definitions of $\Delta_i$ and $L_i$, one gets that
\begin{align}
\Delta_i&=v_2(L_i)-v_2(s(2^n,t))\notag\\
&=v_2\big(s(2^n,i)2^{n(i-t)}\binom{i}{i-t}\big)-v_2(s(2^n,t))\notag\\
&\ge v_2\big(s(2^n,2^l-j)2^{n(2^l-j-2^m+k)}\big)-v_2(s(2^n,2^m-k))\notag\\
&=v_2(s(2^n,2^l-j))+n(2^l-j-2^m+k)-v_2(s(2^n,2^m-k)):=D_{j, k}.\label{2.12}
\end{align}
We divide the proof into the following two cases:

{\sc Case 1.}  $j$ is even. Then with $m$ replaced by $l$ and $k$ by $j$, (\ref{1.3}) gives us that
\begin{align}\label{2.13}
v_2(s(2^n,2^l-j))=2^n-2^l-(n-l)(2^l-j)+l-1-v_2(j),
\end{align}
where $m\le l\le n$, $2\le j\le 2^{l-1}$ and $2\le j\le k-1$ if $l=m$.
Now one considers the following two subcases.

{\sc Case 1.1.} $k$ is even. By (\ref{2.5}) one knows that
\begin{align*}
v_2(s(2^n,2^m-k))=2^n-2^m-(n-m)(2^m-k)+m-1-v_2(k).
\end{align*}
Then together with (\ref{2.12}) and (\ref{2.13}) we obtain that
\begin{align}
D_{j,k}=(l-1)2^l-(m-1)2^m+m(k-1)-l(j-1)+v_2(k)-v_2( j).\label{2.14}
\end{align}

If $l=m$, then $k-j\ge 1$ and $2\le j\le 2^{m-1}$,
and it follows from both of $k$ and
$j$ are even that $k-j\ge 2$, thus by (\ref{2.14}) one can deduce that
\begin{align}\label{2.15}
D_{j,k}= m(k-j)+v_2(k)-v_2(j)\ge2m+v_2(k)-(m-1)=m+v_2(k)+1\ge 4
\end{align}
since $m\ge 2$ and $k$ is even.

If $l\ge m+1$, then by $k\ge 2$ and $2\le j\le 2^{l-1}$ and (\ref{2.14}) we have
\begin{align}
D_{j,k}&\ge (l-1)2^l-(m-1)2^m+m-l(2^{l-1}-1)+v_2(k)-(l-1)\notag\\
&=(l-2)2^{l-1}-(m-1)2^m+m+v_2(k)+1\notag\\
&\ge m+v_2(k)+1\ge 4.\label{2.16}
\end{align}

Hence it follows from (\ref{2.12}) and (\ref{2.15}) and (\ref{2.16})
that $\Delta_i\ge D_{j,k}\ge 4$
when both of $j$ and $k$ are even.
Therefore, claim (\ref{2.8}) is proved in this case.

{\sc Case 1.2.} $k$ is odd. Then $k\ge 3$ and $k-1$ is even.
Again, by (\ref{2.5}) one has
\begin{align}
v_2(s(2^n,2^m-k)) &= 2^n-2^m-(n-m)\Big(2^m-2\Big\lfloor\frac{k}{2}\Big\rfloor\Big)
+ m-2-v_2\Big(\Big\lfloor\frac{k}{2}\Big\rfloor\Big)+(n-1)\notag\\
&= 2^n-2^m-(n-m)(2^m-(k-1))+m-2-v_2\Big(\frac{k-1}{2}\Big)+(n-1)\notag\\
&=v_2(s(2^n,2^m-(k-1)))+n-1.\label{2.17}
\end{align}

If $l=m$ with $(k-1)-j\ge1$ namely $k-j\ge 2$ or $l\ge m+1$, since $k-1$ is even, then
one can apply Case 1.1 to the case $k-1$ and get that
\begin{align}\label{2.18}
v_2(s(2^n,2^l-j))+n(2^l-j-2^m+k-1)-v_2(s(2^n,2^m-(k-1)))=D_{j,k-1}\ge 4.
\end{align}
So, by (\ref{2.12}) together with (\ref{2.17}) and (\ref{2.18}), we derive that
\begin{align}
\Delta_i&\ge v_2(s(2^n,2^l-j))+n(2^l-j-2^m+k)-v_2(s(2^n,2^m-k))\notag\\
&=v_2(s(2^n,2^l-j))+n(2^l-j-2^m+k-1)-v_2(s(2^n,2^m-(k-1)))+1\notag\\
&=D_{j,k-1}+1\ge 5.\notag
\end{align}

On the other hand, if $l=m$ and $k-j=1$, then $i=2^m-k+1$ is even.
Thus, by (\ref{2.7}) and (\ref{2.17}) we have
\begin{align*}
\Delta_i&=\Delta_{2^m-k+1}=v_2(L_{2^m-k+1})-v_2(s(2^n,2^m-k))\\
&= v_2\big(s(2^n,2^m-k+1)2^n\binom{2^m-k+1}{1}\big)-v_2(s(2^n,2^m-k))\\
&=v_2(s(2^n,2^m-(k-1)))+v_2(2^m-k+1)-v_2(s(2^n,2^m-(k-1)))+1\\
&=v_2(2^m-k+1)+1\ge 2.
\end{align*}

Hence, claim (\ref{2.8}) is proved in this case.

{\sc Case 2.} $j$ is odd. Then $j\ge 3$ and $j-1$ is even,
with $m$ replaced by $l$ and $k$ by $j$ in (\ref{2.17}), one gets that
\begin{align}\label{2.19}
v_2(s(2^n,2^l-j))=v_2(s(2^n,2^l-(j-1)))+n-1.
\end{align}
Applying Case 1 to $j-1$ also tells us that,
if $l=m$ with $k-(j-1)\ge 2$ or $l\ge m+1$ then
\begin{align}\label{2.20}
v_2(s(2^n,2^l-(j-1)))+n(2^l-(j-1)-2^m+k)-v_2(s(2^n,2^m-k))=D_{j-1,k}\ge 4.
\end{align}
Since one has $m\le l\le n$ and $k-j\ge 1$ when $l=m$, then
it is clear that either $l=m$ with $k-(j-1)\ge 2$ or $l\ge m+1$.
Thus, by (\ref{2.12}), (\ref{2.19}) and (\ref{2.20}), one obtains that
\begin{align*}
\Delta_i&\ge v_2(s(2^n,2^l-j))+n(2^l-j-2^m+k)-v_2(s(2^n,2^m-k))\\
&=v_2(s(2^n,2^l-(j-1)))+n(2^l-(j-1)-2^m+k)-v_2(s(2^n,2^m-k))-1\\
&=D_{j-1,k}-1\ge 3
\end{align*}
as (\ref{2.8}) claimed.

This finishes the proof of Lemma \ref{lem4}.
\end{proof}

For integers $n$ and $i$ such that $n\ge 2$ and $1\le i\le 2^{n-1}$,
the following lemma reveals a connection between
$v_2(s(2^n,2i-1))$ and $v_2(s(2^n,2i))$.

\begin{lem}\label{lem5}
Let $n$ be an integer with $n\ge 2$. If (\ref{1.3}) holds
for all integers $m$ and even $k$ with $(m,k)\in T_n$,
then for any integer $i$ with $1\le i\le 2^{n-1}$, we have
\begin{align*}
v_2(s(2^n,2i-1))=v_2(s(2^n,2i))+n-1.
\end{align*}
\end{lem}

\begin{proof}
By setting $i=2^{n-1}-t$, we know that showing the truth of Lemma \ref{lem5}
is equivalent to show that
\begin{align}\label{2.21}
v_2(s(2^n,2^n-2t-1))=v_2(s(2^n,2^n-2t))+n-1
\end{align}
holds for all integers $t$ with $0\le t\le 2^{n-1}-1$.
To prove this, we use induction on the integer $t$.
Clearly, (\ref{2.21}) is true when $t=0$,
since $v_2(s(2^{n},2^{n}-1))=n-1$ and $v_2(s(2^{n},2^{n}))=0$.
Now let $1\le t\le 2^{n-1}-1$. Assume that ({\ref{2.21}})
is true for all integers $e$ with $0\le e\le t-1$. In the following
we prove that ({\ref{2.21}}) holds for the integer $t$.

Replacing $n$ by $2^n$ and $k$ by $2^n-2t-1$ in Lemma \ref{lem1}, we get that
\begin{align}
s(2^{n},2^n-2t-1)&=\frac{1}{2}\sum_{i=2^n-2t}^{2^{n}}s(2^{n},i)
\binom{i-1}{i-2^n+2t+1}2^{n(i-2^n+2t+1)}(-1)^{2^{n}-i}\notag\\
& = (2^n-2t-1)2^{n-1} s(2^{n},2^n-2t)+2^{n-1}L,\label{2.22}
\end{align}
where $L=\sum_{i=2^n-2t+1}^{2^{n}}L_i$
and
\begin{align}\label{2.23}
L_i=s(2^{n},i)\binom{i-1}{i-2^n+2t+1}2^{n(i-2^n+2t)}(-1)^{2^{n}-i}.
\end{align}

For any integer $i$ with $2^n-2t+1\le i\le 2^n$,
one can always write $i=2^n-2t+2r-1$
or $i=2^n-2t+2r$ for an integer $r$ with $1\le r\le t$.
Then $0\le t-r\le t-1$. So by the induction assumption, one derives that
$$v_2(s(2^n,2^n-2t+2r-1))=v_2(s(2^n,2^n-2(t-r)-1))=v_2(s(2^n,2^n-2(t-r)))+n-1,$$
which implies immediately that
\begin{align}\label{2.24}
v_2(s(2^{n},2^n-2t+2r-1)2^{n(2r-1)})=v_2(s(2^{n},2^n-2t+2r)2^{2nr})-1.
\end{align}
Hence for any integer $r$ with $1\le r\le t$, if it holds that
\begin{align}\label{2.25}
v_2(s(2^{n},2^n-2t+2r)2^{2nr}) \ge v_2(s(2^{n},2^n-2t))+4,
\end{align}
then by (\ref{2.23}) to (\ref{2.25}), one can deduce
\begin{align*}
v_2(L)&=v_2\Big(\sum_{i=2^n-2t+1}^{2^n}L_i\Big)
\ge \min_{2^n-2t+1\le i\le 2^n}\{v_2(L_i)\}\\
&\ge \min_{2^n-2t+1\le i\le 2^n}\{v_2(s(2^n,i)2^{n(i-2^n+2t)})\}
\ge v_2(s(2^{n},2^n-2t))+3.
\end{align*}
It infers that
\begin{align*}
v_2(2^{n-1}L)>v_2((2^n-2t-1)2^{n-1}s(2^{n},2^n-2t)).
\end{align*}
Thus by (\ref{2.22}) and using the isosceles triangle principle,
one arrives at
$$v_2(s(2^n,2^n-2t-1))=v_2((2^n-2t-1)2^{n-1}s(2^{n},2^n-2t))
=v_2(s(2^n,2^n-2t))+n-1$$
as desired. So to finish the proof of Lemma \ref{lem5}, it remains
to show that (\ref{2.25}) is true. This will be done in what follows.

Since $2\le 2^n-2t\le 2^n-2$,
one can write $2^n-2t=2^m-k$ for some $(m,k)\in T_n$ with $k$ being even.
Then by (\ref{1.3}), one gets that
\begin{align}\label{2.26}
v_2(s(2^n,2^m-k)) = 2^n-2^m-(n-m)(2^m-k)+m-1-v_2(k).
\end{align}

Note that $2^n-2t+2\le 2^n-2t+2r\le 2^n$ since $1\le r\le t$.
If $2^n-2t+2r=2^n$, then $s(2^n,2^n-2t+2r)=s(2^n,2^n)=1$
and $2r=2t=2^n-2^m+k$.
Thus by (\ref{2.26}), together with $2\le m\le n$,
$k\ge 2$ and $k$ being even, one deduces that
\begin{align*}
&v_2(s(2^n,2^n-2t+2r)2^{2nr})-v_2(s(2^n,2^n-2t))=n(2^n-2^m+k)-v_2(s(2^n,2^m-k))\\
&=(n-1)2^n-(m-1)2^m+m(k-1)+v_2(k)+1\ge m+v_2(k)+1\ge 4.
\end{align*}
Hence (\ref{2.25}) is true in this case.

If $2^n-2t+2\le 2^n-2t+2r\le 2^n-2$, then as in the proof of Lemma \ref{lem4},
one may let $2^n-2t+2r:=2^l-j$ for some $(l,j)\in T_n$ with $j$ being even.
However, $2^n-2t=2^m-k$. Hence $l\ge m$ and $2r=2^l-j-(2^n-2t)=2^l-j-2^m+k$,
also notice that $2r=k-j\ge 2$ if $l=m$
and both of $j$ and $k$ are even. Then as the derivation of (\ref{2.12}),
(\ref{2.15}) and (\ref{2.16}), one can obtain that
\begin{align*}
&v_2(s(2^n,2^n-2t+2r)2^{2nr})-v_2(s(2^n,2^n-2t))\\
&=v_2(s(2^n,2^l-j)2^{n(2^l-j-2^m+k)})-v_2(s(2^n,2^m-k))\\
&=v_2(s(2^n,2^l-j))+n(2^l-j-2^m+k)-v_2(s(2^n,2^m-k))=D_{j,k}\ge 4
\end{align*}
as (\ref{2.25}) expects. So (\ref{2.25}) is proved.

This completes the proof of Lemma \ref{lem5}.
\end{proof}

\section{Proof of Theorem \ref{thm1}}

In this section, we present the proof of Theorem \ref{thm1}.\\

\noindent{\it Proof of Theorem \ref{thm1}.} We prove Theorem \ref{thm1}
by induction on $n \geq 2$.

First, let $n=2$. Then $m=2$ and $k\in \{2, 3\}$.
Since $v_2(s(2^2,2^2-2))=v_2(11)=0$ and $v_2(s(2^2,2^2-3))=v_2(6)=1$,
(1.1) is true when $n=2$. So Theorem \ref{thm1} holds for the case $n=2$.
Assume that Theorem \ref{thm1} is true for the $n$ case with $n\ge 2$.
Then (\ref{1.3}) holds for all $(m,k)\in T_{n}$.

Now we prove Theorem \ref{thm1} for the $n+1$ case. Namely, we have to
show that for all $(m,k)\in T_{n+1}$, one has
\begin{align}\label{3.1}
v_2(s(2^{n+1},2^m-k))=2^{n+1}-2^m-(n-m+1)\Big(2^m-2\Big\lfloor\frac{k}{2}\Big\rfloor\Big)
+ m-2-v_2\Big(\Big\lfloor\frac{k}{2}\Big\rfloor\Big)+n\epsilon_k,
\end{align}
where $\epsilon_k=0$ if $k$ is even, and $\epsilon_k=1$ if $k$ is odd.
This will be done in what follows.

Let $(m,k)\in T_{n+1}$. If we can show that (\ref{3.1})
holds for all even integers $k$ with $2\le k\le 2^{m-1}$, then
(\ref{3.1}) is true for all odd integers $k$ with $3\le k\le 2^{m-1}+1$.
Actually, let $k$ be an odd integer.
Then $k-1$ is even and $2\le k-1\le 2^{m-1}$.
So by (3.1), one obtains that
\begin{align}
&v_2(s(2^{n+1},2^m-(k-1)))\notag\\
&=2^{n+1}-2^m-(n-m+1)\Big(2^m-2\Big\lfloor\frac{k-1}{2}\Big\rfloor\Big)
+ m-2-v_2\Big(\Big\lfloor\frac{k-1}{2}\Big\rfloor\Big)\notag\\
&=2^{n+1}-2^m-(n-m+1)(2^m-(k-1))+ m-2-v_2\Big(\frac{k-1}{2}\Big)\notag\\
&=2^{n+1}-2^m-(n-m+1)\Big(2^m-2\Big\lfloor\frac{k}{2}\Big\rfloor\Big)
+ m-2-v_2\Big(\Big\lfloor\frac{k}{2}\Big\rfloor\Big).\label{3.2}
\end{align}
But Lemma \ref{lem5} tells us that
\begin{align}
v_2(s(2^{n+1},2^m-k))=v_2(s(2^{n+1},2^m-k+1))+n.\label{3.3}
\end{align}
It then follows from (\ref{3.2}) and (\ref{3.3}) that
\begin{align*}
v_2(s(2^{n+1},2^m-k))=2^{n+1}-2^m-(n-m+1)\Big(2^m-2\Big\lfloor\frac{k}{2}\Big\rfloor\Big)
+ m-2-v_2\Big(\Big\lfloor\frac{k}{2}\Big\rfloor\Big)+n.
\end{align*}
Therefore (3.1) is true for any odd integer $k$ with $3 \leq k \leq 2^{m-1}+1$.
So in what follows, we need just to
show that (3.1) is true for all even integers $k$ with $2 \leq k \leq 2^{m-1}$.
It will be done in the remaining part of the proof.

In what follows, let $(m,k)\in T_{n+1}$ with $k$ being even, and
we will show the truth of (\ref{3.1}). Since $k$ is even, showing (\ref{3.1})
is equivalent to show the following identity:
\begin{align}\label{3.4}
v_2(s(2^{n+1},2^m-k))=W_{m,k}:=2^{n+1}-2^m-(n-m+1)(2^m-k)+m-1-v_2(k).
\end{align}

Let $m=2$, then $k=2$ and so by (\ref{1.1}), we have
$$
s(2^{n+1},2^m-k)=s(2^{n+1},2)=(2^{n+1}-1)!H(2^{n+1}-1, 1).
$$
Applying the well-known formula $v_2(H(n, 1))=-\lfloor \log_2n\rfloor$
with $n$ replaced by $2^{n+1}-1$, one then gets that
\begin{align*}
&v_2(s(2^{n+1},2^m-k))=v_2(s(2^{n+1},2))=v_2((2^{n+1}-1)!)+v_2(H(2^{n+1}-1, 1))\\
&=2^{n+1}-1-(n+1)-\lfloor\log_2(2^{n+1}-1)\rfloor
=2^{n+1}-2n-2=W_{2,2}=W_{m,k},
\end{align*}
which implies the truth of (\ref{3.4}) when $m=2$.

Now assume that $3\le m\le n+1$. Replacing both of $m$ and $n$ by $2^n$,
$k$ by $2^m-k$ in Lemma \ref{lem2},
one obtains that
$$
s(2^{n+1},2^m-k)=\sum_{i=1}^{2^m-k}g_k(i),
$$
where
$$g_k(i):=s(2^n,i)s_{2^n}(2^n,2^m-k-i).$$
Let
\begin{align*}
S_1:=\sum_{i\in A}g_k(i),\  S_2:=\sum_{i\in B}g_k(i),\  S_3:=\sum_{i\in C}g_k(i),
\end{align*}
where the sets $A, B, C$ are defined by
\begin{align*}
A:=\left\{\begin{array}
{ll}\{2^{m-1},2^{m-1}-k\} & {\rm if}\ 2 \leq k \leq 2^{m-2},\\
\{2^{m-2},2^{m}-k-2^{m-2}\} & {\rm if}\ m\ge 4\ {\rm and}\ 2^{m-2}+2 \leq k \leq 2^{m-1}-2,\\
\{2^{m-2}\} & {\rm if}\ k=2^{m-1}
\end{array}\right.
\end{align*}
and
\begin{align*}
B:=([2,2^m-k]\cap 2\mathbb{Z})\setminus A, \ C:=[1, 2^m-k]\cap (1+2\mathbb{Z}).
\end{align*}
Then
\begin{align}
s(2^{n+1},2^m-k)=S_1+S_2+S_3.\notag
\end{align}
We claim that the following statements hold:\\
(I). $v_2(S_1)=W_{m,k}$.\\
(II). $v_2(S_2)\ge W_{m,k}+1$.\\
(III). $v_2(S_3)\ge W_{m,k}+1$.\\
It then follows from the isosceles triangle principle and
the claims (I) to (III) that
$$
v_2(s(2^{n+1},2^m-k))=v_2(S_1+S_2+S_3)=v_2(S_1)=W_{m,k}
$$
as (\ref{3.4}) required. So to finish the proof of Theorem 1.1,
it remains to show the truth of the claims (I) to (III).

{\it Proof of claim (I).}
At first, we prove claim (I): $v_2(S_1)=W_{m,k}$.

If $2\le k\le 2^{m-2}$, then $A=\{2^{m-1},2^{m-1}-k\}$ and
\begin{align}
S_1&=g_k(2^{m-1})+g_k(2^{m-1}-k)\notag\\
&=s(2^n,2^{m-1})s_{2^n}(2^n,2^{m-1}-k)+s(2^n,2^{m-1}-k)s_{2^n}(2^n,2^{m-1}).\notag
\end{align}
By Lemma 2.3, one can write
$s_{2^n}(2^n,2^{m-1}-k)=s(2^n,2^{m-1}-k)+L_1$
and $s_{2^n}(2^n,2^{m-1})=s(2^n,2^{m-1})+L_2$,
where
$$
L_1=\sum_{i=2^{m-1}-k+1}^{2^n}s(2^n,i)2^{n(i-2^{m-1}+k)}\binom{i}{i-2^{m-1}+k},
$$
and
$$
L_2=0\ {\rm if}\ m=n+1,\  L_2=\sum_{i=2^{m-1}+1}^{2^n}s(2^n,i)2^{n(i-2^{m-1})}\binom{i}{i-2^{m-1}}\ {\rm if}\ 3\le m\le n .
$$
Then
\begin{align}
g_k(2^{m-1})=s(2^n,2^{m-1})s(2^n,2^{m-1}-k)+s(2^n,2^{m-1})L_1,\label{3.5}
\end{align}
\begin{align}
g_k(2^{m-1}-k)=s(2^n,2^{m-1})s(2^n,2^{m-1}-k)+s(2^n,2^{m-1}-k)L_2 \label{3.6}
\end{align}
and
\begin{align}
S_1=2s(2^n,2^{m-1})s(2^n,2^{m-1}-k)+s(2^n,2^{m-1})L_1+s(2^n,2^{m-1}-k)L_2.\label{3.7}
\end{align}
Since $3\le m\le n+1$ and $2\le k\le 2^{m-2}$,
one has $2\le 2^{m-1}-k\le 2^{n}-2$ and $4\le 2^{m-1}\le 2^n$.
It then follows from the induction assumption and replacing $t$ by
$2^{m-1}-k$ and $2^{m-1}$ in Lemma \ref{lem4} that
\begin{align}
v_2(L_1)=v_2(s_{2^n}(2^n,2^{m-1}-k)-s(2^n,2^{m-1}-k))\ge v_2(s(2^n,2^{m-1}-k))+2 \label{3.8}
\end{align}
and
\begin{align}
v_2(L_2)=v_2(s_{2^n}(2^n,2^{m-1})-s(2^n,2^{m-1}))\ge v_2(s(2^n,2^{m-1}))+2,\label{3.9}
\end{align}
respectively. Furthermore, by the induction assumption, we have
\begin{align}
v_2(s(2^n,2^{m-1}))=v_2(s(2^n,2^m-2^{m-1}))=2^n-2^{m}-(n-m)2^{m-1}\label{3.10}
\end{align}
and
\begin{align}
v_2(s(2^n,2^{m-1}-k))&=2^n-2^{m-1}-(n-m+1)(2^{m-1}-k)+m-2-v_2(k).\label{3.11}
\end{align}
By (\ref{3.5}), (\ref{3.8}), (\ref{3.10}) and (\ref{3.11}) and using the isosceles triangle principle,
one then deduces that
\begin{align}
v_2(g_k(2^{m-1}))&=v_2(s(2^n,2^{m-1}))+v_2(s(2^n,2^{m-1}-k))\notag\\
&=2^{n+1}-2^m-(n-m+1)(2^{m}-k)+m-2-v_2(k)=W_{m,k}-1.\label{3.12}
\end{align}
Likewise, by using  (\ref{3.6}), (\ref{3.9}) to (\ref{3.11}) and the isosceles triangle principle,
we get that
\begin{align}
v_2(g_k(2^{m-1}-k))&=v_2(s(2^n,2^{m-1}))+v_2(s(2^n,2^{m-1}-k))=W_{m,k}-1.\label{3.13}
\end{align}
Moreover, (\ref{3.8}) together with (\ref{3.9}) infers that
\begin{align}
 v_2(s(2^n,2^{m-1})L_1+s(2^n,2^{m-1}-k)L_2)>v_2(2s(2^n,2^{m-1})s(2^n,2^{m-1}-k)).\label{3.14}
\end{align}
Hence by (\ref{3.7}) and (\ref{3.14}) and again using the isosceles triangle principle, one arrives at
\begin{align}
v_2(S_1)&=v_2(2s(2^n,2^{m-1})s(2^n,2^{m-1}-k))\notag\\
&=v_2(s(2^n,2^{m-1}))+v_2(s(2^n,2^{m-1}-k))+1=W_{m,k}\notag
\end{align}
as desired. So claim (I) is true when $2\le k\le 2^{m-2}$.

If $2^{m-2}+2 \leq k \leq 2^{m-1}-2$ with $4\le m\le n+1$,
then $A=\{2^{m-2},2^{m}-k-2^{m-2}\}$, and so
\begin{align}
S_1&=g_k(2^{m-2})+g_k(2^{m}-k-2^{m-2})\notag\\
&=s(2^n,2^{m-2})s_{2^n}(2^n,2^{m}-k-2^{m-2})+s(2^n,2^{m}-k-2^{m-2})s_{2^n}(2^n,2^{m-2}).\label{3.15}
\end{align}
By Lemma 2.3, we have
\begin{align}
s_{2^n}(2^n,2^{m}-k-2^{m-2})&=s(2^n,2^{m}-k-2^{m-2})+L_1^{'}\label{3.16}
\end{align}
and
\begin{align}
s_{2^n}(2^n,2^{m-2})&=s(2^n,2^{m-2})+L_2^{'},\label{3.17}
\end{align}
where
$$
L_1^{'}=\sum_{i=2^{m}-k-2^{m-2}+1}^{2^n}
s(2^n,i)2^{n(i-2^{m}+k+2^{m-2})}\binom{i}{i-2^{m}+k+2^{m-2}}
$$
and
$$
L_2^{'}=\sum_{i=2^{m-2}+1}^{2^n}s(2^n,i)2^{n(i-2^{m-2})}\binom{i}{i-2^{m-2}}.
$$
Since $2\le 2^{m}-k-2^{m-2}=2^{m-1}-(k-2^{m-2})\le 2^n-2$ and $2\le 2^{m-2}\le 2^{n-1}$,
in Lemma \ref{lem4}, replacing $t$ by $2^{m}-k-2^{m-2}$ and $2^{m-2}$, respectively, one gets that
\begin{align}
v_2(s_{2^n}(2^n,2^{m}-k-2^{m-2}))= v_2(s(2^n,2^{m}-k-2^{m-2})), \label{3.18}
\end{align}
\begin{align}
v_2(L_1^{'})&=v_2(s_{2^n}(2^n,2^{m}-k-2^{m-2})-s(2^n,2^{m}-k-2^{m-2}))\notag\\
&\ge v_2(s(2^n,2^{m}-k-2^{m-2}))+2, \label{3.19}
\end{align}
and
\begin{align}
v_2(s_{2^n}(2^n,2^{m-2}))= v_2(s(2^n,2^{m-2})), \label{3.20}
\end{align}
\begin{align}
v_2(L_2^{'})=v_2(s_{2^n}(2^n,2^{m-2})-s(2^n,2^{m-2}))\ge v_2(s(2^n,2^{m-2}))+2, \label{3.21}
\end{align}
respectively. It then follows from (\ref{3.15}) to (\ref{3.21}) and the isosceles triangle principle that
\begin{align}
&v_2(S_1)\notag\\
&=v_2(2s(2^n,2^{m-2})s(2^n,2^{m}-k-2^{m-2})
+s(2^n,2^{m-2})L_1^{'}+s(2^n,2^{m}-k-2^{m-2})L_2^{'})\notag\\
&=v_2(2s(2^n,2^{m-2})s(2^n,2^{m}-k-2^{m-2}))\notag\\
&=v_2(s(2^n,2^{m-2}))+v_2(s(2^n,2^{m}-k-2^{m-2}))+1.\label{3.22}
\end{align}
Also, by (\ref{3.18}), one has
\begin{align}
v_2(g_k(2^{m-2}))&=v_2(s(2^n,2^{m-2})s_{2^n}(2^n,2^{m}-k-2^{m-2}))\notag\\
&=v_2(s(2^n,2^{m-2}))+v_2(s(2^n,2^{m}-k-2^{m-2})),\label{3.23}
\end{align}
and by (\ref{3.20}), we have
\begin{align}
v_2(g_k(2^{m}-k-2^{m-2}))&=v_2(s(2^n,2^{m}-k-2^{m-2})s_{2^n}(2^n,2^{m-2}))\notag\\
&=v_2(s(2^n,2^{m-2}))+v_2(s(2^n,2^{m}-k-2^{m-2})).\label{3.24}
\end{align}
Noticing that $2\le k-2^{m-2}\le 2^{m-2}-2$, then the inductive hypothesis tells that
\begin{align}\label{3.25}
v_2(s(2^n,2^{m-2}))=v_2(s(2^n,2^{m-1}-2^{m-2}))=2^n-2^{m-1}-(n-m+1)2^{m-2}
\end{align}
and
\begin{align}
&v_2(s(2^n,2^{m}-k-2^{m-2}))=v_2(s(2^n,2^{m-1}-(k-2^{m-2})))\notag\\
&=2^n-2^{m-1}-(n-m+1)(2^{m-1}-(k-2^{m-2}))+m-2-v_2(k).\label{3.26}
\end{align}
So by (\ref{3.22}) to (\ref{3.26}) one obtains that
\begin{align}
v_2(g_k(2^{m-2}))&=v_2(g_k(2^{m}-k-2^{m-2}))\notag\\
&=2^{n+1}-2^m-(n-m+1)(2^{m}-k)+m-2-v_2(k)=W_{m,k}-1\label{3.27}
\end{align}
and
\begin{align*}
v_2(S_1)&=2^{n+1}-2^m-(n-m+1)(2^{m}-k)+m-2-v_2(k)+1=W_{m,k}
\end{align*}
as expected. Thus claim (I) holds when $2^{m-2}+2 \leq k \leq 2^{m-1}-2$
with $4\le m\le n+1$.

If $k=2^{m-1}$, then $A=\{2^{m-2}\}$ and
$S_1=g_{2^{m-1}}(2^{m-2})=s(2^n,2^{m-2})s_{2^n}(2^n,2^{m-2})$.
By (\ref{3.17}), (\ref{3.21}), (\ref{3.25}) and the isosceles triangle principle,
also noticing that $W_{m,k}=W_{m,2^{m-1}}=2^{n+1}-2^m-(n-m+1)2^{m-1}$,
we derive that
\begin{align}
v_2(S_1)&=v_2(g_{2^{m-1}}(2^{m-2}))=v_2(s(2^n,2^{m-2})^2+s(2^n,2^{m-2})L_2^{'})\notag\\
&=2v_2(s(2^n,2^{m-2}))=2^{n+1}-2^{m}-(n-m+1)2^{m-1}=W_{m,k}\label{3.28}
\end{align}
as one desires. This concludes the proof of claim (I).

The identities (\ref{3.12}), (\ref{3.13}), (\ref{3.27}) and (\ref{3.28}) also
tell us that
\begin{align}
v_2(g_k(i))\ge W_{m,k}-1\label{3.29}
\end{align}
holds for all integers $i\in A$.

{\it Proof of claim (II).}
Let $i\in B$. If $i > 2^n$ or $2^m-k-i > 2^n$,
then it is clear that $g_k(i)=s(2^n,i)s_{2^n}(2^n,2^m-k-i)=0$.
Therefore $S_2$ can be rewritten as
\begin{align*}
S_2=\sum_{i\in \widetilde{B}}g_k(i),
\end{align*}
where $\widetilde{B}=B\cap [2^m-k-2^n,2^n]$.
Note that $B$ is nonempty (since $2^m-k\in B$). So if $\widetilde{B}$ is empty then
$S_2=0$, thus claim (II) follows.

Now assume that $\widetilde{B}$ is nonempty.
Then one can divide $\widetilde{B}$ into the disjoint union:
$$\widetilde{B}=B_1\cup B_2\cup B_3\cup B_4,$$
where
\begin{align*}
B_1:=\Big[2,2^{m-1}-\frac{k}{2}\Big)\cap\widetilde{B},\ \
B_2:= (2^{m-1}-\frac{k}{2},2^m-k)\cap\widetilde{B},
\end{align*}
$$B_3:=\Big\{2^{m-1}-\frac{k}{2}\Big\}\cap\widetilde{B}, \ \ B_4:=\{2^m-k\}\cap\widetilde{B}.$$
At least one of $B_1, B_2, B_3$ and $B_4$ is nonempty. So
\begin{align}
S_2=\sum_{j=1}^4\sum_{i\in B_j}g_k(i).\label{3.30}
\end{align}

First, we handle $B_3$.
Suppose that $B_3$ is nonempty. Then for $i\in B_3$, we show that
\begin{align}
v_2(g_k(i))\ge W_{m, k}+1.\label{3.31}
\end{align}
Since $i\in B_3$, one has $i=2^{m-1}-\frac{k}{2}$ and so $i=2^m-k-i$.
We must have $k\ne 2^{m-1}$. Otherwise, $k=2^{m-1}$ implies that
$i=2^{m-2}$ which contradicts with the definition of $B$ when $k=2^{m-1}$.
This concludes that $2\le k\le 2^{m-1}-2$. It infers that $v_2(k)\le m-2$.
Since $2\le 2^{m-1}-\frac{k}{2}\le 2^n-2$ and noticing that $\frac{k}{2}$ is even,
by the inductive hypothesis and Lemma \ref{lem4}, one gets that
\begin{align}
v_2(s_{2^n}(2^n, 2^m-k-i))=v_2\Big(s_{2^n}\Big(2^n,2^{m-1}-\frac{k}{2}\Big)\Big)
=v_2\Big(s\Big(2^n,2^{m-1}-\frac{k}{2}\Big)\Big) \label{3.32}
\end{align}
and
\begin{align}
v_2(s(2^n, i))&=v_2\Big(s\Big(2^n,2^{m-1}-\frac{k}{2}\Big)\Big)\notag\\
&=2^n-2^{m-1}-(n-m+1)\Big(2^{m-1}-\frac{k}{2}\Big)+m-2-v_2\Big(\frac{k}{2}\Big).\label{3.33}
\end{align}
Hence by (\ref{3.32}) and (\ref{3.33}), one obtains that
\begin{align}
v_2(g_k(i))-W_{m,k}&=v_2(s(2^n,i)s_{2^n}(2^n,2^m-k-i))-W_{m,k}\notag\\
&=2v_2\Big(s\Big(2^n,2^{m-1}-\frac{k}{2}\Big)\Big)-W_{m,k}\notag\\
&=m-1-v_2(k)\ge 1.\label{3.34}
\end{align}
Then (\ref{3.31}) follows immediately. Thus (\ref{3.31}) holds when $i\in B_3$.

Now we show that if $B_4$ is nonempty, then
(\ref{3.31}) also holds when $i\in B_4$. Let $i\in B_4$. Then $i=2^m-k$ and
\begin{align}\label{3.35}
v_2(s_{2^n}(2^n, 2^m-k-i))=v_2(s_{2^n}(2^n,0))=v_2\Big(\frac{(2^{n+1}-1)!}{(2^n-1)!}\Big)=2^n-1.
\end{align}
If $i=2^m-k=2^n$, then by $3\le m\le n+1$ and $2\le k\le 2^{m-1}$,
we get that $m=n+1$ and $k=2^n$. Since $s(2^n,i)=s(2^n,2^n)=1$, $n\ge 2$
and $W_{m, k}=W_{n+1, 2^n}=0$, one has
\begin{align}
v_2(g_k(i))=v_2(s(2^n,i)s_{2^n}(2^n, 2^m-k-i))=2^n-1\ge 3=W_{m,k}+3\label{3.36}
\end{align}
as (\ref{3.31}) expected. If $2\le 2^m-k\le 2^n-2$, then replacing $m-1$ by $m$
and $\frac{k}{2}$ by $k$ in (\ref{3.33}) together with (\ref{3.35}), one arrives at
\begin{align}
v_2(g_k(i))-W_{m,k}&=v_2(s(2^{n},i)s_{2^n}(2^n,2^m-k-i))-W_{m,k} = 2^m-k-1\ge 1.\label{3.37}
\end{align}
Hence (\ref{3.31}) is true when $i\in B_4$.

Consequently, assume that $B_1$ is nonempty. Then $B_2$ is nonempty too.
Actually, since $B_1$ is nonempty, one picks $i\in B_1$. Then it is
easy to see that $2^m-k-i\in B_2$. So $B_2$ is nonempty.
Furthermore, it is clear that $\# B_1=\#B_2$. We define a map
$$\tau: B_1\rightarrow B_2$$
by $\tau(i):=2^m-k-i$ for any $i\in B_1$.
Evidently, $\tau(2^m-k-j)=j$ for any $j\in B_2$. That is, $\tau$
is injective. So $\tau$ is a bijective map from $B_1$ to $B_2$.
It then follows that
\begin{align}
\sum_{i\in B_1}g_k(i)+\sum_{i\in B_2}g_k(i)=\sum_{i\in B_1}g_k(i)+\sum_{i\in B_1}g_k(\tau(i))
=\sum_{i\in B_1}(g_k(i)+g_k(2^m-k-i)).\label{3.38}
\end{align}

On the other hand, for any $i\in B_1$ one has $2\le i\le 2^{m-1}-\frac{k}{2}-1\le 2^n-2$ and
$2^m-k-i\in B_2$, it infers that $4\le 2^{m-1}-\frac{k}{2}+1\le2^m-k-i\le 2^n$.
Thus by the induction assumption and Lemma \ref{lem4}, one deduces that
\begin{align*}
v_2(g_k(i))=v_2(s(2^n,i))+v_2(s_{2^n}(2^n,2^m-k-i))=v_2(s(2^n,i))+v_2(s(2^n,2^m-k-i))
\end{align*}
and
\begin{align*}
v_2(g_k(2^m-k-i))&=v_2(s(2^n,2^m-k-i))+v_2(s_{2^n}(2^n,i))\\
&=v_2(s(2^n,2^m-k-i))+v_2(s(2^n,i)).
\end{align*}
Hence
\begin{align}
v_2(g_k(i))=v_2(g_k(2^m-k-i)).\label{3.39}
\end{align}
So for any $i\in B_1$, if we can show that
\begin{align}
v_2(g_k(i))\ge W_{m,k},\label{3.40}
\end{align}
then (\ref{3.39}) together with (\ref{3.40}) implies that
\begin{align}
v_2(g_k(i)+g_k(2^m-k-i))\ge v_2(g_k(i))+1\ge W_{m,k}+1.\label{3.41}
\end{align}
Thus it follows from (\ref{3.30}), (\ref{3.31}), (\ref{3.38}) and (\ref{3.41}) that
\begin{align}
v_2(S_2)&=v_2(\sum_{i\in B_1}(g_k(i)+g_k(2^m-k-i))+\sum_{i\in B_3}g_k(i)+\sum_{i\in B_4}g_k(i))\notag\\
&\ge \min\{\min_{i\in B_1}\{v_2(g_k(i)+g_k(2^m-k-i))\}, W_{m,k}+1\}\notag\\
&\ge  W_{m,k}+1\label{3.42}
\end{align}
as claim (II) desired. Also, by (\ref{3.31}), (\ref{3.39}) and (\ref{3.40}), one can conclude that
$v_2(g_k(i))\ge W_{m,k}$ is still true for all integers $i\in \widetilde{B}$.
So to finish the proof of claim (II), it remains to show that (\ref{3.40}) is true
for any $i\in B_1$. This will be done in what follows.

Let $i\in B_1$. Then we can write $i=2^{l_1}-j_1$ and $2^m-k-i=2^{l_2}-j_2$,
where $(l_1,j_1)\in T_{m-1}$ and $(l_2,j_2)\in T_{m}$
with both of $j_1$ and $j_2$ being even.
Notice that $2^m-k-i\in B_2$, and so $l_2\ge m-1\ge l_1$.
Hence $l_2\in\{m-1,m\}$. Also, we have
\begin{align}
2^{m}-k=2^{l_1}-j_1+2^{l_2}-j_2.\label{3.43}
\end{align}
Since $2\le k\le 2^{m-1}$, $2\le j_1\le 2^{l_1-1}$ and $2\le j_2\le 2^{l_2-1}$,
by (\ref{3.43}) one then derives that
\begin{align}
v_2(k)-v_2(j_1)-v_2(j_2)&=v_2(2^{m}-k)-v_2(j_1)-v_2(j_2)\notag\\
&=v_2(2^{l_1}-j_1+2^{l_2}-j_2)-v_2(j_1)-v_2(j_2)\notag\\
&\ge \min\{v_2(2^{l_1}-j_1),v_2(2^{l_2}-j_2)\}-v_2(j_1)-v_2(j_2)\notag\\
&= \min\{v_2(j_1),v_2(j_2)\}-v_2(j_1)-v_2(j_2)\notag\\
&= \min\{-v_2(j_1),-v_2(j_2)\}\notag\\
&=-\max\{v_2(j_1), v_2(j_2)\}.\label{3.44}
\end{align}
Moreover, it follows from $2\le l_1\le m-1\le n$ and the induction assumption that
\begin{align}
v_2(s(2^n,i))&=v_2(s(2^n,2^{l_1}-j_1)= 2^n-2^{l_1}-(n-l_1)(2^{l_1}-j_1)+l_1-1-v_2(j_1).\label{3.45}
\end{align}
Consider the following two cases.

{\sc Case 1.} $l_2=n+1$. Then one gets that $l_2=m=n+1$ since $l_2\le m\le n+1$.
But $2^m-k-i\in B_2$ tells us that $2^m-k-i=2^{l_2}-j_2=2^{n+1}-j_2\le 2^n$.
So by $2\le j_2\le 2^{l_2-1}=2^n$, one derives that $j_2=2^n$.
Thus $2^m-k-i=2^{n+1}-2^n=2^n$ and so $i=2^m-k-2^n=2^n-k=2^{m-1}-k$.
It implies that $2+2^{m-2}\le k\le 2^{m-1}-2$
since $i>0$ and $i\ne 2^{m-1}-k$ when $2\le k\le 2^{m-2}$.
Hence $2\le i\le 2^{m-2}-2=2^{n-1}-2$, which infers that $l_1\le n-1$.
Then by (\ref{3.45}) together with the fact that $s(2^n,2^n)=1$
and noticing that $W_{m,k}=W_{n+1,k}=n-v_2(k)$, we deduce that
\begin{align}
v_2(g_k(i))-W_{m,k}&=v_2(s(2^n,i))+v_2(s(2^n,2^m-k-i))-W_{m,k}=v_2(s(2^n,i))-W_{m,k}\notag\\
&=2^n+(l_1-n-1)2^{l_1}+(j_1-1)(n-l_1)-1+v_2(k)-v_2(j_1)\notag\\
&\ge 2^n+(l_1-n-1)2^{l_1}+n-l_1-1\notag\\
&= 2^n+(l_1-n-1)(2^{l_1}-1)-2:= D_{l_1}\label{3.46}
\end{align}
since $2\le j_1\le 2^{l_1-1}\le 2^{n-2}$ and $i=2^{l_1}-j_1=2^n-k$
implying that $v_2(k)=v_2(j_1)$.

Now we show that $D_{l_1}\ge 0$, then (\ref{3.40}) follows immediately.
For this purpose, we introduce an auxiliary function $h_t(x):=(x-t)(2^x-1)$. Then
$h_t'(x)=2^x(1-(t-x)\ln2)-1$, and so $h_t'(x)<0$ for $x\le t-2$.
Thus $h_t(x)$ is decreasing when $x\le t-2$. It infers that
$$
D_{l_1}=2^n+h_{n+1}(l_1)-2\ge 2^n+h_{n+1}(n-1)-2=0
$$
since $2\le l_1\le n-1$. Hence (\ref{3.40}) is proved in this case.

{\sc Case 2.} $2\le l_2\le n$. Since $j_2$ is even, by the induction assumption, one has
\begin{align}
v_2(s(2^n,2^m-k-i))&=v_2(s(2^n,2^{l_2}-j_2))\notag\\
&= 2^n-2^{l_2}-(n-l_2)(2^{l_2}-j_2)+l_2-1-v_2(j_2).\label{3.47}
\end{align}
By (\ref{3.43}), (\ref{3.45}) and (\ref{3.47}), one obtains that
\begin{align}
v_2(g_k(i))-W_{m,k}&=v_2(s(2^n,i))+v_2(s(2^n,2^m-k-i))-W_{m,k}\notag\\
&=2^m+(l_1-m)2^{l_1}+(l_2-m)2^{l_2}+j_1(m-l_1-1)+j_2(m-l_2-1)\notag\\
&\ \ \ \ +l_1+l_2-m-1+v_2(k)-v_2(j_1)-v_2(j_2):=D_{l_1,l_2}.\label{3.48}
\end{align}
In what follows, we show that $D_{l_1,l_2}\ge 0$, which concludes the proof of
(\ref{3.40}). Since $l_2\in\{m-1,m\}$,
we divide this into following two subcases.

{\sc Case 2.1.} $l_2=m-1$. Then one has $i=2^{l_1}-j_1$ with $2\le j_1\le 2^{l_1-1}\le 2^{m-2}$
and $2^m-k-i=2^{m-1}-j_2$ with $2\le j_2\le 2^{m-2}$.
Since $2\le l_1\le m-1$, if we can show that
\begin{align}
\max\{v_2(j_1), v_2(j_2)\}\le m-3,\label{3.49}
\end{align}
then by (\ref{3.49}) together with (\ref{3.48}) and (\ref{3.44}), one can deduce that
\begin{align}
D_{l_1,l_2}&=D_{l_1,m-1}\notag\\
&=2^{m-1}+(l_1-m)2^{l_1}+j_1(m-l_1-1)+l_1-2+v_2(k)-v_2(j_1)-v_2(j_2)\notag\\
&\ge 2^{m-1}+(l_1-m)2^{l_1}+2(m-l_1-1)+l_1-2-\max\{v_2(j_1), v_2(j_2)\}\notag\\
&= 2^{m-1}+(l_1-m)(2^{l_1}-1)-1\notag\\
&\ge 2^{m-1}-(2^{m-1}-1)-1\notag\\
&=0\notag
\end{align}
as desired.
Since $2\le j_1\le 2^{m-2}$ and $2\le j_2\le 2^{m-2}$, in oder
to prove that (\ref{3.49}) holds, it suffices to show that $j_1\ne 2^{m-2}$ and $j_2\ne 2^{m-2}$.
This will be done in what follows.

If $2\le k\le 2^{m-2}$, then we must have $l_1=m-1$. Otherwise, one has $2\le l_1\le m-2$,
from which and (\ref{3.43}) follows that
$$2^m-k\le 2^{m-2}-2+2^{m-1}-2=2^{m-1}+2^{m-2}-4\le 2^m-k-4.$$
This is a contradiction. So $l_1=m-1$. By (\ref{3.43}), one then gets that $j_1+j_2=k$.
It infers that $j_1\ne 2^{m-2}$ and $j_2\ne 2^{m-2}$ since $j_1\ge 2$ and $j_2\ge 2$.
Hence (\ref{3.49}) follows.

If $2+2^{m-2}\le k\le 2^{m-1}$, then by the definition of $B_1$, one has
$i\ne 2^{m-2}$ and $2^m-k-i=2^{m-1}-j_2\ne 2^{m-2}$. Thus $j_2\ne 2^{m-2}$.
Since $2\le j_1\le 2^{l_1-1}\le 2^{m-2}$, one gets that $l_1=m-1$ when $j_2=2^{m-2}$,
and so $i=2^{l_1}-j_1=2^{m-1}-2^{m-2}=2^{m-2}$. But $i\ne 2^{m-2}$. Hence $j_1\ne 2^{m-2}$.
We then conclude that (\ref{3.49}) is true.

Therefore, $D_{l_1,l_2}\ge 0$ is proved in this case.

{\sc Case 2.2.} $l_2=m$. By (\ref{3.44}) one obtains that
\begin{align}
v_2(k)-v_2(j_1)-v_2(j_2)\ge -\max\{v_2(j_1), v_2(j_2)\}\ge -(m-1) \label{3.50}
\end{align}
since $2\le j_1\le 2^{l_1-1}\le 2^{m-2}$ and $2\le j_2\le 2^{l_2-1}=2^{m-1}$.

If $l_1=m-1$, then it follows from (\ref{3.43}) that $k=j_1+j_2-2^{m-1}\le 2^{m-2}$.
This infers that $2^m-k-i=2^{m}-j_2\ne 2^{m-1}$, and so $j_2\ne 2^{m-1}$.
Hence $2\le j_2\le 2^{m-1}-2$.
Together with (\ref{3.48}) and (\ref{3.50}), one then arrives at
\begin{align*}
D_{l_1,l_2}&=D_{m-1,m}=2^{m-1}-j_2+m-2+v_2(k)-v_2(j_1)-v_2(j_2)\\
&\ge 2^{m-1}-(2^{m-1}-2)+m-2-(m-1)=1.
\end{align*}

If $2\le l_1\le m-2$, then by (\ref{3.48}), (\ref{3.50}), $j_1\ge 2$ and $j_2\le 2^{m-1}$,
we derive that
\begin{align*}
D_{l_1,l_2}&=D_{l_1,m}\\
&=2^m+(l_1-m)2^{l_1}+j_1(m-l_1-1)-j_2+l_1-1+v_2(k)-v_2(j_1)-v_2(j_2)\\
&\ge 2^m+(l_1-m)2^{l_1}+2(m-l_1-1)-2^{m-1}+l_1-1-(m-1)\\
&=2^{m-1}+(l_1-m)(2^{l_1}-1)-2\\
&\ge 2^{m-1}-2(2^{m-2}-1)-2\\
&=0
\end{align*}
as expected. Hence $D_{l_1,l_2}\ge 0$ is proved. 
Thus (\ref{3.40}) holds for any $i\in B_1$.

This completes the proof of claim (II).

{\it Proof of claim (III).}
Let $i\in C$. If $i > 2^n$ or $2^m-k-i > 2^n$,
then $g_k(i)=s(2^n,i)s_{2^n}(2^n,2^m-k-i)=0$.
Since $k$ is even and $C$ contains only odd integers,
we can rewrite $S_3$ as:
$$
S_3=\sum_{i\in \widetilde{C}}g_k(i),
$$
where
$\widetilde{C}:=C\cap [2^m-k-2^n,2^n]=C\cap [2^m-k-2^n+1,2^n-1].$

Since $2^{m-1}-k+1\in \widetilde{C}$, $\widetilde{C}$ is nonempty.
Now let $i\in \widetilde{C}$. We will show that the following inequality holds:
\begin{align}
v_2(g_k(i))\ge W_{m,k}+1.\label{3.51}
\end{align}
It then follows that
$$
v_2(S_3)=v_2\Big(\sum_{i\in \widetilde{C}}g_k(i)\Big)
\ge \min_{i\in \widetilde{C}}\{v_2(g_k(i))\}\geq W_{m,k}+1
$$
as claim (III) asserted. So to finish the proof of claim (III),
it is enough to prove the truth of (\ref{3.51}). To do so,
we first let $i=1$. Since $i\in \tilde C$ is odd, we have $2^m-k-2^n+1\le 1$,
from which one can deduce that $1\le 2^m-k-1\le 2^n-1$. Thus by the induction assumption
and Lemma \ref{lem4}, we get that
\begin{align}
v_2(s_{2^n}(2^n,2^m-k-i))&=v_2(s_{2^n}(2^n,2^m-k-1))=v_2(s(2^n,2^m-k-1))\notag\\
&=2^n-2^m-(n-m)(2^m-k)+m+n-2-v_2(k).\label{3.52}
\end{align}
But
\begin{align}
v_2(s(2^n,i))=v_2(s(2^n,1))=v_2((2^n-1)!)=2^n-n-1.\label{3.53}
\end{align}
So together with the hypothesis $k\le 2^{m-1}$ and $m\ge 3$, (\ref{3.52}) and (\ref{3.53}) tells us that
\begin{align*}
v_2(g_k(i))-W_{m,k}&=v_2(s(2^n,i))+v_2(s_{2^n}(2^n,2^m-k-i))-W_{m,k}\\
&=2^m-k-2\ge 2^{m-1}-2\ge 2
\end{align*}
as (\ref{3.51}) desired. So (\ref{3.51}) is true when $i=1$.

Now let $3\leq i\leq 2^n-1$. We claim that
\begin{align}\label{3.54}
v_2(s(2^n,i+1))\geq v_2(s(2^n,i-1))-2n+4.
\end{align}
Evidently, for $i\in \widetilde{C}$ with $i\ge 3$, one has $2^m-k-i\in \widetilde{C}$
that infers that $2\le 2^m-k-(i-1)\le 2^n$.
Then it follows from the inductive hypothesis,
Lemma \ref{lem4} and claim (\ref{3.54}) that
\begin{align}
v_2(g_k(i))&=v_2(s(2^n,i))+v_2(s_{2^n}(2^n,2^m-k-i))\notag\\
&=v_2(s(2^n,i))+v_2(s(2^n,2^m-k-i))\notag\\
&=v_2(s(2^n,i+1))+n-1+v_2(s(2^n,2^{m}-k-i+1))+n-1\notag\\
&\ge v_2(s(2^n,i-1))+v_2(s(2^n,2^{m}-k-(i-1)))+2\notag\\
&=v_2(s(2^n,i-1))+v_2(s_{2^n}(2^n,2^{m}-k-(i-1)))+2\notag\\
&=v_2(g_k(i-1))+2. \label{3.55}
\end{align}
However, $i-1$ is even and $\max\{2, 2^m-k-2^n\}\le i-1\le \min\{2^m-k-2,2^n-2\}$
since $i\in \widetilde{C}$. Hence $i-1\in A\cup B_1\cup B_2\cup B_3$.
By the proof of claims (I) and (II), we know that
$v_2(g_k(j))\ge W_{m,k}-1$ if $j\in A$ and
$v_2(g_k(j))\ge W_{m,k}$ if $j\in B_1\cup B_2\cup B_3$. Thus
$$v_2(g_k(i-1))\ge W_{m,k}-1.$$
Together with (\ref{3.55}) one then arrives at the expected inequality (\ref{3.51}).

So to finish the proof of (\ref{3.51}), it remains to show that claim (\ref{3.54}) is true.
This will be done in what follows.
First, if $i=2^n-1$, then $v_2(s(2^n,i+1))=v_2(s(2^n,2^n))=0$.
By the inductive hypothesis one has $v_2(s(2^n,i-1))=v_2(s(2^n,2^n-2))=n-2$.
Since $n\ge 2$, one arrives at
$$
v_2(s(2^n,i+1))-v_2(s(2^n,i-1))=-n+2=n-2-2n+4\ge -2n+4
$$
as (\ref{3.54}) claimed.

Now suppose that $3\le i\le 2^n-3$. Then $4\le i+1\le 2^n-2$. Since $i+1$ is even,
one may let $i+1=2^l-j$ with $(l,j)\in T_n$ and $j$ being even.
By the induction assumption one obtains that
\begin{align}\label{3.56}
v_2(s(2^n,i+1))= v_2(s(2^n,2^l-j))=2^n-2^l-(n-l)(2^l-j)+l-1-v_2(j).
\end{align}
Also, one has $i-1=2^l-(j+2)$.
If $j+2\le 2^{l-1}$, then $2\le j\le 2^{l-1}-2$ with $l\ge 3$. So $v_2(j)\le l-2$.
Again, by the induction assumption, one gets that
\begin{align}\label{3.57}
v_2(s(2^n,i-1))&= v_2(s(2^n,2^l-(j+2)))\notag\\
&=2^n-2^l-(n-l)(2^l-(j+2))+l-1-v_2(j+2).
\end{align}
Then it follows from (\ref{3.56}), (\ref{3.57}), $1\le v_2(j)\le l-2$ and $l\ge 3$ that
\begin{align}
v_2&(s(2^n,i+1))-v_2(s(2^n,i-1))=-2(n-l)+v_2(j+2)-v_2(j)\notag\\
&\ge-2(n-l)-(l-2)+1=-2n+l+3\ge -2n+6. \notag
\end{align}
On the other hand, if $j+2\ge 2^{l-1}+2$,
then one can deduce that $j=2^{l-1}$
since $(l,j)\in T_n$ together with the fact that $j$ being even implies that $ j\le 2^{l-1}$.
So $i+1=2^{l-1}$ and $i-1=2^{l-1}-2$. But $i+1\ge 4$, so $l\ge 3$.
Thus by the inductive hypothesis, one has
\begin{align}\label{3.58}
v_2(s(2^n,i+1))=v_2(s(2^n,2^l-2^{l-1}))=2^n-2^{l}-(n-l)2^{l-1}
\end{align}
and
\begin{align}\label{3.59}
v_2(s(2^n,i-1))= v_2(s(2^n,2^{l-1}-2))=2^n-2^{l-1}-(n-l+1)(2^{l-1}-2)+l-3.
\end{align}
Since $l\ge 3$, putting (\ref{3.58}) and (\ref{3.59}) together gives us that
\begin{align*}
v_2(s(2^n,i+1))-v_2(s(2^n,i-1))=-2n+l+1\ge -2n+4
\end{align*}
as claim (\ref{3.54}) asserted. This finishes the proof of claim (\ref{3.54}).
Hence inequality (\ref{3.51}) is proved. Thus the proof of claim (III) is complete.

This concludes the proof of Theorem \ref{thm1}.
\hfill$\Box$

\section{Proof of Theorem \ref{thm2}}

We show the truth of Theorem \ref{thm2} in this section.\\

\noindent{\it Proof of Theorem \ref{thm2}.}
If $n=1$, then $k\in \{1,2\}$. But $s(3,2)=3$ and
$s(2,1)=s(3,3)=s(2,2)=1$. So Theorem 1.4 is true when $n=1$.
In the following, we let $n\ge 2$.

By the recurrence relation for the Stirling numbers of the first kind, we know that
\begin{align}\label{4.1}
s(2^n+1,k+1)=2^n s(2^n,k+1)+s(2^n,k).
\end{align}
Thus for any integer $k$ with $1\le k\le 2^n$, if we can show that
\begin{align}\label{4.2}
v_2(s(2^n,k+1))> v_2(s(2^n,k))-n,
\end{align}
then using the isosceles triangle principle together with (\ref{4.1}), one can arrive at
$$
v_2(s(2^n+1,k+1))=v_2(2^ns(2^n,k+1)+s(2^n,k))=v_2(s(2^n,k))
$$
as (\ref{1.4}) desired. Therefore, to finish the proof of Theorem \ref{thm2},
it suffices to show the truth of (\ref{4.2}).
Since $v_2(s(2^n,2^n-1))=n-1$, $v_2(s(2^n,2^n-2))=n-2$,
$s(2^n,2^n+1)=0$ and $s(2^n,2^n)=1$,
(\ref{4.2}) is obviously true when $k\in \{2^n-2,2^n-1,2^n\}$.
So in what follows, we let $1\le k\le 2^n-3$.

If $k$ is odd, then by Theorem \ref{thm1} and Lemma \ref{lem5}, one can deduce that
$$v_2(s(2^n,k))=v_2(s(2^n,k+1))+n-1,$$
which implies that (\ref{4.2}) holds when $k$ is odd.

If $k$ is even, then $k+1$ is odd, and so
$$v_2(s(2^n,k+1))=v_2(s(2^n,k+2))+n-1.$$
But with $i$ replaced by $k+1$ in claim (\ref{3.54}) gives us that
$$v_2(s(2^n,k+2))\ge v_2(s(2^n,k))-2n+4.$$
Hence $v_2(s(2^n,k+1))\ge v_2(s(2^n,k))-n+3.$
Thus (\ref{4.2}) is true when $k$ is even.

This finishes the proof of Theorem \ref{thm2}.
\hfill$\Box$

\section{Proofs of Corollaries \ref{cor2} and \ref{cor3}}

This section is dedicated to the proofs of Corollaries \ref{cor2} and \ref{cor3}.
We begin with the proof of Corollary \ref{cor2}. \\

\noindent{\it Proof of Corollary \ref{cor2}.}
If $n=1$, then $k=1$ or $k=2$. Since $s(2,1)=s(2,2)=1$,
(\ref{1.7}) is true when $n=1$.
Now let $n\ge 2$. Then
$$v_2(s(2^n,1))-v_2(s(2^n,2^n))=2^n-n-1\ge 1$$
and
$$v_2(s(2^n,1))-v_2(s(2^n,2^n-1))=2^n-n-1-(n-1)=2^n-2n\ge 0.$$
So (\ref{1.7}) is proved when $k=2^n$ and $k=2^n-1$.

If $1\le k\le 2^n-2$, then let $k=2^l-j$ with $(l,j)\in T_n$.
By (\ref{1.3}) one obtains that
\begin{align}
&v_2(s(2^n,1))-v_2(s(2^n,k))=v_2(s(2^n,1))-v_2(s(2^n,2^l-j))\notag\\
&=2^l+(n-l)\Big(2^l-2\Big\lfloor\frac{j}{2}\Big\rfloor\Big)
-n-l+1+v_2\Big(\Big\lfloor\frac{j}{2}\Big\rfloor\Big)-(n-1)\epsilon_j:=D_j,\label{3.60}
\end{align}
where $\epsilon_j=0$ if $j$ is even, and $\epsilon_j=1$ if $j$ is odd.
Consider the following two cases.

For the case that $j$ is even, one has $2\le j\le 2^{l-1}$. Since $2\le l\le n$, by (\ref{3.60}) we have
\begin{align*}
D_j&=2^l+(n-l)(2^l-j)-n-l+v_2(j)\\
&=(n-l)(2^l-j-1)+2^l-2l+v_2(j)\\
&\ge (n-l)(2^{l-1}-1)+2^l-2l+1\\
&\ge 1,
\end{align*}
the last inequality is due to the reason that $2^l-2l$ is increasing when $l\ge 2$.

For the case that $j$ is odd, one gets that $3\le j\le 2^{l-1}+1$ and $\lfloor\frac{j}{2}\rfloor=\frac{j-1}{2}$.
Again, by $2\le l\le n$ and (\ref{3.60}) we derive that
\begin{align*}
D_j&=2^l+(n-l)(2^l-j+1)-n-l+v_2(j-1)-(n-1)\\
&=(n-l)(2^l-j-1)+2^l-3l+v_2(j-1)+1\\
&\ge (n-l)(2^{l-1}-2)+2^l-3l+2\\
&\ge 2^l-3l+2\\
&\ge 0.
\end{align*}
Now one can conclude that $D_j\ge 0$. Hence $v_2(s(2^n,k))\le v_2(s(2^n,1))$.
Therefore (\ref{1.7}) is true when $1\le k\le 2^n-2$.

This completes the proof of Corollary \ref{cor2}.
\hfill$\Box$\\

Finally, we present the proof of Corollary \ref{cor3} as the conclusion of this paper.\\

\noindent{\it Proof of Corollary \ref{cor3}.}
Let $n$ and $k$ be positive integers such that $k\le 2^n$.
By (\ref{1.1}) we know that
\begin{align}
H(2^n,k)=\frac{s(2^n+1,k+1)}{(2^n)!}.\label{5.2}
\end{align}
Then by (\ref{5.2}) and Theorem \ref{thm2} together with Corollary \ref{cor2}, one deduces that
\begin{align*}
v_2(H(2^n,k))&=v_2(s(2^n+1,k+1))-v_2((2^n)!)\\
&=v_2(s(2^n,k))-v_2((2^n)!)\\
&\le v_2(s(2^n,1))-v_2((2^n)!)\\
&=v_2(s(2^n+1,2))-v_2((2^n)!)\\
&=v_2(H(2^n,1))\\
&=-n
\end{align*}
as (\ref{1.8}) expected, the last step is because $v_2(\frac{1}{k})\ge -n+1$ for all integers $k$
with $1\le k\le 2^n-1$ and $v_2(\frac{1}{2^n})=-n$.

This concludes the proof of Corollary \ref{cor3}.
\hfill$\Box$
\\

\bibliographystyle{amsplain}

\end{document}